\documentclass[11pt]{scrartcl}

\usepackage[utf8]{inputenc}

\usepackage{config}
\usepackage{titling}

\usepackage[a4paper, top=3cm, bottom=2cm, left=2.5cm, right=2.5cm, includefoot]{geometry}

\usepackage[inline]{enumitem}

\setlist[itemize]{topsep=0pt,partopsep=0pt,itemsep=0pt,parsep=0pt}
\setlist[itemize,1]{label={\small\textbullet}}
\setlist[itemize,2]{label={\tiny\textbullet}}
\setlist[itemize,3]{label=$\cdot$}
\setlist[enumerate]{topsep=0pt,partopsep=0pt,itemsep=0pt,parsep=0pt}
\setlist[enumerate,1]{label=\roman*)}
\setlist[enumerate,2]{label=\alph*)}
\setlist[enumerate,3]{label=\arabic*)}

\hypersetup{
	colorlinks=true,
	linkcolor=AO!65!black,
	citecolor=AO!65!black,
	urlcolor=AppleGreen!65!black,
	bookmarksopen=true,
	bookmarksnumbered,
	bookmarksopenlevel=2,
	bookmarksdepth=3
      }

\title{A note on the 2-factor Hamiltonicity Conjecture}
\predate{}
\date{}
\postdate{}

\preauthor{}
\DeclareRobustCommand{\authorthing}{
	\begin{center}
		Maximilian Gorsky\thanks{This work was supported by the Institute for Basic Science (IBS-R029-C1).}~~\!\footnote{\href{mailto:m.gorsky@pm.me}{m.gorsky@pm.me}}\\
		Discrete Mathematics Group, Institute for Basic Science (IBS), Daejeon, South Korea

        \medskip
        
            Theresa Johanni\\
		Technische Universit\"at Berlin, Germany

        \medskip

            Sebastian Wiederrecht\footnote{\href{mailto:sebastian.wiederrecht@gmail.com}{sebastian.wiederrecht@gmail.com}}\\
		School of Computing, KAIST, South Korea
\end{center}}
\author{\authorthing}
\postauthor{}

\begin{document}
\maketitle

\begin{abstract}
The 2-factor Hamiltonicity Conjecture by Funk, Jackson, Labbate, and Sheehan [JCTB, 2003] asserts that all cubic, bipartite graphs in which all 2-factors are Hamiltonian cycles can be built using a simple operation starting from $K_{3,3}$ and the Heawood graph.

We discuss the link between this conjecture and matching theory, in particular by showing that this conjecture is equivalent to the statement that the two exceptional graphs in the conjecture are the only cubic braces in which all 2-factors are Hamiltonian cycles, where braces are connected, bipartite graphs in which every matching of size at most two is contained in a perfect matching.
In the context of matching theory this conjecture is especially noteworthy as $K_{3,3}$ and the Heawood graph are both strongly tied to the important class of Pfaffian graphs, with $K_{3,3}$ being the canonical non-Pfaffian graph and the Heawood graph being one of the most noteworthy Pfaffian graphs.

Our main contribution is a proof that the Heawood graph is the only Pfaffian, cubic brace in which all 2-factors are Hamiltonian cycles.
This is shown by establishing that, aside from the Heawood graph, all Pfaffian braces contain a cycle of length four, which may be of independent interest.
\end{abstract}

\section{Introduction}\label{sec:intro}

Characterising graphs containing a cycle that spans all of its vertices, known as a \emph{Hamiltonian cycle}, is a classic problem in graph theory, which is known to be $\NP$-complete \cite{Karp1972Reducibility}.
A good characterisation of these graphs, which are called \emph{Hamiltonian}, is therefore unlikely to exist.
Nonetheless, the study of these graphs remains a very popular topic (see \cite{Gould1991Updating,Gould2014Recent}).

We focus on a problem first discussed in \cite{Diwan2002Disconnected}, concerning graphs in which all \emph{2-factors}, which are 2-regular subgraphs containing all vertices of the graph, are Hamiltonian cycles.
Such graphs are called \emph{2-factor Hamiltonian}.\footnote{This class was apparently discussed earlier by Sheehan who posed an unpublished conjecture involving them in the '80s \cite{AbreuL2017Factors}. This conjecture was later proven as the main theorem of \cite{FunkJLS20032Factor}.}
In \cite{Diwan2002Disconnected}, a construction for an infinite family of cubic, bipartite, 2-factor Hamiltonian graphs is given by starting with $K_{3,3}$ and the Heawood graph (see \Cref{fig:heawood}) and repeatedly performing the \emph{star product}.
The star product $(G_1,v_1) * (G_2,v_2)$ of two graphs $G_1$ and $G_2$ is obtained by taking two vertices $v_1 \in V(G_1)$ and $v_2 \in V(G_2)$ of degree 3, deleting both, and joining the three neighbours of $v_1$ in $G_1$ to the three neighbours of $v_2$ in $G_2$ via a matching\footnote{A \emph{matching} is a set of disjoint edges in a graph and it is called \emph{perfect} if all vertices of the graph are contained in some edge of the matching.} of size 3.
This matching is called the \emph{principal 3-edge cut} of $(G_1,v_1) * (G_2,v_2)$.
When applied outside of the context of cubic graphs, the star product exhibits some unexpected behaviour in relation to 2-factor Hamiltonian graphs, which we discuss in \Cref{app:starproduct}.

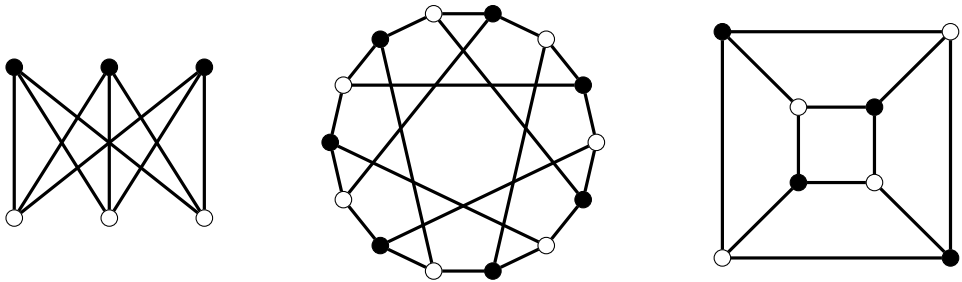
\begin{figure}[ht]
     \centering
            \raisebox{20pt}{
            \begin{tikzpicture}[scale=1]
			\node (V1) at (0,0) [draw, circle, scale=0.6] {};
			\node (V2) at (1.25,0) [draw, circle, scale=0.6] {};
			\node (V3) at (2.5,0) [draw, circle, scale=0.6] {};
			
			\node (U1) at (0,2) [draw, circle, scale=0.6, fill] {};
			\node (U2) at (1.25,2) [draw, circle, scale=0.6, fill] {};
			\node (U3) at (2.5,2) [draw, circle, scale=0.6, fill] {};
			
			\path
			(V1) edge[very thick] (U1)
            (V1) edge[very thick] (U2)
            (V1) edge[very thick] (U3)
			(V2) edge[very thick] (U1)
            (V2) edge[very thick] (U2)
            (V2) edge[very thick] (U3)
			(V3) edge[very thick] (U1)
			(V3) edge[very thick] (U2)
            (V3) edge[very thick] (U3)
			;
			\end{tikzpicture}
            }
            \qquad \quad
    		\begin{tikzpicture}[scale=1]
			
			\node (V0) at (0:0) [draw=none] {};
			
			\foreach\i in {1,3,5,7,9,11,13}
			{
				\node (V\i) at ($(V0)+({(360/14 * \i)}:1.75)$) [draw, circle, scale=0.6, fill, label={}] {};
			}
			
			\foreach\i in {2,4,6,8,10,12,14}
			{
				\node (V\i) at ($(V0)+({(360/14 * \i)}:1.75)$) [draw, circle, scale=0.6, label={}] {};
			}
			
			\foreach\i in {1,2,3,4,5,6,7,8,9,10,11,12,13}
			{
				\pgfmathtruncatemacro\iplus{\i+1}
				\path (V\i) edge[very thick] (V\iplus);
			}
			\path (V14) edge[very thick] (V1);
			
			\path
				(V1) edge[very thick] (V6)
				(V2) edge[very thick] (V11)
				(V3) edge[very thick] (V8)
				(V4) edge[very thick] (V13)
				(V5) edge[very thick] (V10)
				(V7) edge[very thick] (V12)
				(V9) edge[very thick] (V14)
			;
			
			\end{tikzpicture}
            \qquad \quad
            \raisebox{5pt}{
            \begin{tikzpicture}[scale=1]
			\node (V1) at (0,0) [draw, circle, scale=0.6] {};
			\node (V2) at (3,0) [draw, circle, scale=0.6, fill] {};
			\node (V3) at (3,3) [draw, circle, scale=0.6] {};
			\node (V4) at (0,3) [draw, circle, scale=0.6, fill] {};
			
			\node (U1) at (1,1) [draw, circle, scale=0.6, fill] {};
			\node (U2) at (2,1) [draw, circle, scale=0.6] {};
			\node (U3) at (2,2) [draw, circle, scale=0.6, fill] {};
			\node (U4) at (1,2) [draw, circle, scale=0.6] {};
			
			\path
			(V1) edge[very thick] (V2)
			(V2) edge[very thick] (V3)
			(V3) edge[very thick] (V4)
			(V4) edge[very thick] (V1)
			(U1) edge[very thick] (U2)
			(U2) edge[very thick] (U3)
			(U3) edge[very thick] (U4)
			(U4) edge[very thick] (U1)
			(V1) edge[very thick] (U1)
			(V2) edge[very thick] (U2)
			(V3) edge[very thick] (U3)
			(V4) edge[very thick] (U4)
			;
			\end{tikzpicture}
            }
            \caption{The graph depicted to the left is $K_{3,3}$.
            The drawing in the middle depicts the Heawood graph.
            On the very right a drawing of the cube is given.}
            \label{fig:heawood}
\end{figure}

In \cite{FunkJLS20032Factor}, Funk et al.\ focus on characterising $k$-regular, bipartite, 2-factor Hamiltonian graphs and show that these can only exist if $k \in \{ 2,3 \}$.
With the graphs for $k=2$ being rather dull, all that remains is to characterise cubic, bipartite, 2-factor Hamiltonian graphs and they conjecture that this should be the class previously identified by Diwan \cite{Diwan2002Disconnected}.

\begin{conjecture}[Funk, Jackson, Labbate, and Sheehan \cite{FunkJLS20032Factor}]\label{con:2FH}
    All cubic, bipartite, 2-factor Hamiltonian graphs can be constructed from $K_{3,3}$ and the Heawood graph via the star product.
\end{conjecture}

Somewhat unconventionally, our interest in this conjecture comes from the vantage point of matching theory.
In this context $K_{3,3}$ and the Heawood graph are usually seen in opposing roles, as they are both key figures in the study of \emph{Pfaffian graphs}.
An undirected graph $G$ with a perfect matching is called \emph{Pfaffian} if there exists an orientation of its edges such that for every perfect matching $M$ of $G$ and every cycle $C$ that alternates between edges in $M$ and $E(G) \setminus M$, the number of edges in $C$ oriented in the same way is odd, for either direction of traversal.

Despite the somewhat esoteric definition, Pfaffian graphs have a rich variety of applications (see \cite{McCuaig2004Polyas,Thomas2006Survey}) and in particular, bipartite, Pfaffian graphs are the key ingredient in the solution of an old problem posed by \Polya \cite{Polya1913Aufgabe} concerning the computation of the permanent of a matrix \cite{RobertsonST1999Permanents,McCuaig2004Polyas}.
Finding a useful characterisation of bipartite, Pfaffian graphs that could be applied to develop a polynomial time recognition algorithm was a long-standing open problem.
An early and influential characterisation by Little \cite{Little1975Characterization} prominently features $K_{3,3}$, though no algorithm could be derived from this result.
A \emph{bisubdivision} of a graph $G$ is obtained by replacing the edges of $G$ with internally disjoint paths of odd length.

\begin{theorem}[Little \cite{Little1975Characterization}]
    A bipartite graph $G$ with a perfect matching is Pfaffian if and only if it does not contain a bisubdivision $H$ of $K_{3,3}$ such that $G - H$ has a perfect matching.
\end{theorem}

The resolution of the problem of characterising Pfaffian, bipartite graphs was ultimately found by McCuaig, and Robertson, Seymour, and Thomas \cite{McCuaigRST1997Permanents}, who published their solutions independently \cite{RobertsonST1999Permanents,McCuaig2004Polyas}.
We will use their characterisation for our main result and thus we take a moment to define the necessary concepts here.

A \emph{cut} $\Cut{X}$ around a set of vertices $X \subseteq V(G)$ in a graph $G$ is the set of edges in $G$ with one endpoint in $X$ and the other in $\Compl{X} = V(G) \setminus X$.
If every edge of a graph is contained in a perfect matching of that graph, we call the graph \emph{matching covered}.
In a matching covered graph $G$, a cut $\Cut{X}$ such that $|\Cut{X} \cap M| = 1$ holds for every perfect matching $M$ of $G$ is called \emph{tight} and $\Cut{X}$ is called \emph{trivial} if $|X| = 1$ or $|\Compl{X}| = 1$.
Given a tight cut $\Cut{X}$ in a matching covered graph $G$, the two \emph{tight cut contractions} derived from $\Cut{X}$ are constructed by contracting $X$, respectively $\Compl{X}$, into a single vertex and removing all loops and parallel edges resulting from this.

In \cite{Lovasz1987Matching} \Lovasz introduced a procedure called the \emph{tight cut decomposition} in which tight cut contractions are repeatedly performed on a matching covered graph until we are left with a list of graphs that do not have non-trivial tight cuts.
Graphs without non-trivial tight cuts are called \emph{braces} if they are bipartite and \emph{bricks} if they are non-bipartite.
The contents of the list at the end of the procedure are called the bricks and braces of the graph that was decomposed.
For any given graph this list is uniquely determined and does not depend on the choice of the tight cuts in the decomposition procedure and the decomposition can be found in polynomial time \cite{Lovasz1987Matching}.

For many problems it suffices to consider bricks and braces, as several properties are preserved under tight cut contractions.
This for example also holds for being Pfaffian.

\begin{theorem}[Vazirani and Yannakakis \cite{VaziraniY1989Pfaffian}]
    A matching covered graph is Pfaffian if and only if all of its bricks and braces are Pfaffian.
\end{theorem}

Characterising Pfaffian bricks remains an important open problem, but we are now ready to present the characterisation of Pfaffian braces that resolves the problem we presented earlier.
Let $G_1, G_2, G_3$ be three bipartite graphs, such that their pairwise intersection is a 4-cycle $C$, and we have $V(G_i) \setminus V(C) \neq \emptyset$ for all $i \in [3]$.
Further, let $S \subseteq E(C)$ be some subset of the edges of $C$.
The \emph{trisum of $G_1, G_2, G_3$ at $C$} is a graph $\bigcup_{i=1}^3 G_i - S$.

\begin{theorem}[McCuaig \cite{McCuaig2004Polyas}, Robertson, Seymour, and Thomas \cite{RobertsonST1999Permanents}]\label{thm:pfaffian}
    A brace is Pfaffian if and only if it is the Heawood graph or it can be constructed from planar braces by repeated use of the trisum operation.
\end{theorem}

Thus $K_{3,3}$ is the most prominent example of a non-Pfaffian graph, the Heawood graph is a prominent Pfaffian graph, and both of these graphs are braces.
Next we reformulate \Cref{con:2FH} to focus on braces.
To facilitate this, we use a definition of braces that is a little easier to handle than the non-existence of non-trivial tight cuts.

Let $k$ be a positive integer and let $G$ be a graph on at least $2k + 2$ vertices, then $G$ is called \emph{$k$-extendable} if every matching of size $k$ is contained in a perfect matching of $G$.
It turns out that by using Lemma 1.4 from \cite{Lovasz1987Matching} and Theorem 2.2 from \cite{Plummer1986Matching}, we can equivalently characterise graphs as braces if they are either $C_4$ or have at least six vertices and are 2-extendable.
For later use, we state a helpful result from the article in which Plummer introduces extendability.

\begin{theorem}[Plummer \cite{Plummer1980$n$extendable}]\label{thm:extendabilitybasics}
    For all positive integers $k$, any $k$-extendable graph is $(k-1)$-extendable and $(k+1)$-connected.\footnote{A connected graph $G$ is \emph{$k$-connected} if no set $S \subseteq V(G)$ with $|S| \leq k-1$ exists such that $G - S$ is disconnected.}
\end{theorem}

For a cubic, 3-connected, bipartite graph $G$ being a brace is actually equivalent to being cyclically 4-connected \cite{GorskySW2023Matching}\footnote{A graph $G$ is called \emph{cyclically 4-connected} if for each cut $F$ of order less than four, $G - F$ contains at most one component that is not a tree.} and in \cite{FunkJLS20032Factor} Funk et al.\ point out that \Cref{con:2FH} can in fact be reduced to cyclically 4-connected graphs via a result by Labbate \cite{Labbate20013cut}.
With some effort, one can use these results to show that \Cref{con:2FH} is in fact equivalent to the following.

\begin{conjecture}\label{con:main}
    The Heawood graph and $K_{3,3}$ are the only 2-factor Hamiltonian, cubic braces.
\end{conjecture}

The path to proving the correspondence between \Cref{con:2FH} and \Cref{con:main} is complicated by the fact that in \cite{Labbate20013cut} Labbate uses a notion that is distinct from 2-factor Hamiltonicity, but happens to coincide with it in this particular setting and cyclic 4-connectivity is never explicitly mentioned.
Beyond this, the result in \cite{GorskySW2023Matching} requires 3-connectivity, which is also not explicitly mentioned when Funk et al.\ discuss the result from \cite{Labbate20013cut} in \cite{FunkJLS20032Factor}.
Thus we provide a short proof of this reduction using basic matching theory in \Cref{sec:braces}.

To further reduce \Cref{con:2FH}, we note that it is also known that \Cref{con:2FH} can be restricted to graphs of girth at least 6 \cite{Labbate2002Characterizing,FunkJLS20032Factor}, where the \emph{girth} of a graph is the length of a shortest cycle within it. 
We show the following in \Cref{sec:pfaffian} and thus eliminate Pfaffian graphs as possible counterexamples to \Cref{con:2FH}.

\begin{theorem}\label{thm:main}
    All Pfaffian braces other than the Heawood graph have girth 4.
\end{theorem}

\begin{corollary}\label{cor:main}
    The Heawood graph is the only Pfaffian, 2-factor Hamiltonian, cubic brace.
\end{corollary}

We note that \Cref{thm:main} stands in stark contrast to results by Gvozdjak and Ne\v{s}et\v{r}il \cite{GvozdjakN1996High} who were able to construct a bipartite, $k$-extendable graph $G_{n,k}$ with girth at least $n$ for any given positive integers $k$ and $n$.\footnote{See the third remark in the conclusion of \cite{GvozdjakN1996High} for the fact that these graphs are bipartite.}
Additionally, the main result in \cite{Diwan2002Disconnected} implies that no planar, Hamiltonian graph can be 2-factor Hamiltonian.
Since all planar graphs with perfect matchings are Pfaffian \cite{Kasteleyn1963Dimer}, \Cref{cor:main} thus further extends this result as well.

Ultimately, \Cref{con:2FH} can be narrowed down to the following, drawing from Lemma 3.3 in \cite{FunkJLS20032Factor} and the fact that $K_{3,3}$ has girth 4.

\begin{conjecture}
    All non-Pfaffian, cubic braces $G$ of girth at least 6, with $|V(G)| \equiv 2 \pmod{4}$, are not 2-factor Hamiltonian.
\end{conjecture}

\section{Reduction to braces}\label{sec:braces}

Here we give a proof of the fact that \Cref{con:2FH} can be reduced to braces that is largely independent of the work in \cite{Labbate20013cut}, \cite{FunkJLS20032Factor}, and \cite{GorskySW2023Matching}, with the exception of the following simple result from \cite{FunkJLS20032Factor} (see Lemma 3.3).

\begin{lemma}[Funk, Jackson, Labbate, and Sheehan \cite{FunkJLS20032Factor}]\label{lem:3con}
    All 2-factor Hamiltonian, cubic, bipartite graphs are 3-connected.
\end{lemma}

The following old theorem by \Konig will also prove useful.

\begin{theorem}[\Konig \cite{Konig1916Ueber}]\label{thm:konig}
    For every positive integer $k$, the edges of every $k$-regular, bipartite graph can be partitioned into $k$ pairwise disjoint perfect matchings. 
\end{theorem}

This immediately implies that all $k$-regular, bipartite, connected graphs are matching covered, allowing us to perform the tight cut decomposition procedure on them, and in particular, \Cref{thm:konig} also tells us that all of these graphs have 2-factors.
We first use this to reduce the problem of characterising 2-factor Hamiltonian, cubic, bipartite graphs to 3-connected graphs.

Concerning the tight cut decomposition procedure for 3-connected, cubic, bipartite graphs, McCuaig provides us with two very useful lemmas.
We call a matching $F$ in a graph $G$ \emph{induced}, if no two endpoints of distinct edges of $F$ are adjacent in $G$.

\begin{lemma}[McCuaig \cite{McCuaig2000Even}]\label{lem:tightcutshape}
    Let $G$ be a 3-connected, cubic, bipartite graph.
    A non-trivial cut in $G$ is an induced matching consisting of 3 edges if and only if it is tight.
\end{lemma}

\begin{lemma}[McCuaig \cite{McCuaig2000Even}]\label{lem:tightcutpreserve}
    Let $G$ be a 3-connected, cubic, bipartite graph with a non-trivial tight cut $\Cut{X}$.
    Then both tight cut contractions associated with $\Cut{X}$ are cubic, 3-connected, and bipartite.
\end{lemma}

This lets us prove that tight cut contractions preserve 2-factor Hamiltonicity in our setting.

\begin{lemma}\label{lem:tightcut2FH}
    Let $G$ be a 3-connected, cubic, bipartite graph with a non-trivial tight cut $\Cut{X}$.
    Then $G$ is 2-factor Hamiltonian if and only if the two tight cut contractions associated with $\Cut{X}$ are 2-factor Hamiltonian.
\end{lemma}
\begin{proof}
    Let $G_1$ and $G_2$ be the two tight cut contractions associated with $\Cut{X}$, such that $X \cup \{ c_1 \} = V(G_1)$ and $\Compl{X} \cup \{ c_2 \} = V(G_2)$, meaning that $c_1$ and $c_2$ are the two contraction vertices.
    Using \Cref{lem:tightcutshape}, we let $\{ e, f, g \} = \Cut{X}$.
    For both $i \in [2]$, we let $e_i$ be the unique edge incident to $c_i$ in $G_i$ that shares an endpoint with $e$ and we define $f_i$ and $g_i$ analogously.

    Suppose that $G$ is 2-factor Hamiltonian but at least one of $G_1$ and $G_2$ is not 2-factor Hamiltonian.
    Both of these graphs contain 2-factors according to \Cref{thm:konig}, since they are 3-connected, cubic, and bipartite thanks to \Cref{lem:tightcutpreserve}. 
    Without loss of generality, we can therefore assume that $G_1$ contains a 2-factor $F$ that is not a Hamiltonian cycle.
    We may further assume, again without loss of generality, that $e_1,f_1 \in E(F)$ are the two edges in $F$ that are incident to $c_1$.
    Note that there exists a cycle $C$ in $F$ that does not use the edges $e_1,f_1$ and thus we have $C \subseteq G$.
    As observed earlier, \Cref{thm:konig} tells us that there exists a 2-factor $F'$ in $G_2$ using the edges $e_2,f_2$.
    We can now observe that $(F \cup F' \cup \{ e,f \}) - \{ c_1, c_2 \}$ is a 2-factor in $G$ that contains at least 2 components, contradicting the 2-factor Hamiltonicity of $G$.
    Thus both $G_1$ and $G_2$ must be 2-factor Hamiltonian.

    For the other direction, suppose that $G_1$ and $G_2$ are 2-factor Hamiltonian, but $G$ is not.
    First, observe that according to \Cref{lem:tightcutpreserve} and \Cref{thm:konig} the graph $G_1$ has a perfect matching and thus the vertex set of $G_1$ can be partitioned into two sets of equal size, such that all edges of $G_1$ have one endpoint in each of these two sets.
    Therefore no 2-factor of $G$ contains a 2-factor covering exactly the vertices of $G[X]$, as $V(G_1) = X \cup \{ c_1 \}$, and such a 2-factor of $X$ could be split into two perfect matchings of the vertices of $X$, according to \Cref{thm:konig}.
    An analogous argument shows that such a 2-factor cannot exist for $\Compl{X}$ either.
    Thus each 2-factor of $G$ contains exactly two edges of $\Cut{X}$, as this cut contains three edges.

    Suppose now that $F$ is a 2-factor of $G$ that is not a Hamiltonian cycle, which must exist, since $G$ must have a 2-factor and is not 2-factor Hamiltonian.
    Let $C$ be the component of $F$ that contains the two edges of $\Cut{X}$ and without loss of generality let $e,f \in E(C)$ be these two edges.
    Once more without loss of generality there exists a non-empty subgraph $H \subseteq \InducedG{G}{X}$ of $F$ such that $H \cup C$ contains all vertices of $X$.
    We note that the endpoints of $e$ and $f$ are not adjacent according to \Cref{lem:tightcutshape} and thus $C - \Compl{X}$ is a path of length at least two.
    Therefore, if we let $C'$ be the result of contracting the path $C - X$ within $C$ into the vertex $c_1$, then $C'$ is guaranteed to be a cycle.
    However, this implies that $H \cup C'$ is a 2-factor of $G_1$ that is not a Hamiltonian cycle, contradicting the 2-factor Hamiltonicity of $G_1$ and completing our proof.
\end{proof}

Thus to characterise 2-factor Hamiltonian, cubic, bipartite graphs it suffices to characterise 2-factor Hamiltonian, cubic braces.
However, this is not directly connected to \Cref{con:2FH} yet and we will have to also briefly analyse the star product before we can conclude that \Cref{con:2FH} can be reduced to braces.

\begin{lemma}\label{lem:starproductmatching}
    The principal 3-edge cut $F$ of the star product $(G_1,v_1) * (G_2,v_2)$ of two cubic, bipartite graphs is an induced matching consisting of 3 edges.
\end{lemma}
\begin{proof}
    The fact that $|F| = 3$ is immediate.
    For the fact that the endpoints $u,w$ of any two distinct edges in $F$ are non-adjacent, first note that any such edge must have both endpoints either in $V(G_1)$ or in $V(G_2)$ according to the definition of the star product.
    Suppose that $u,w \in V(G_1)$.
    Both $u$ and $w$ are adjacent with $v_1$ and thus the edge $uw$ creates a cycle of length three, contradicting the fact that $G_1$ is bipartite.
    For an analogous reason no such edge may exist if $u,w \in V(G_2)$.
\end{proof}

We further note the following nice correspondence between the star product and tight cuts.

\begin{observation}\label{obs:starproductcorrespondence}
    Let $G$ be a 3-connected, cubic, bipartite graph with a non-trivial tight cut $\Cut{X}$.
    If $G_1$ and $G_2$ are the two tight cut contractions associated with $\Cut{X}$, then $G$ is a star product of $G_1$ and $G_2$.
\end{observation}

This leads us nicely into the reduction of \Cref{con:2FH} to \Cref{con:main}.

\begin{theorem}
    \Cref{con:2FH} holds if and only if \Cref{con:main} is true.
\end{theorem}
\begin{proof}
    Suppose first that \Cref{con:2FH} holds and that there exists another 2-factor Hamiltonian, cubic brace $B$ that is neither the Heawood graph nor $K_{3,3}$.
    Accordingly $B$ must be the star product of two 3-connected, cubic, bipartite graphs.
    But \Cref{lem:starproductmatching} and \Cref{lem:tightcutshape} imply that the principal 3-edge cut involved is a non-trivial tight cut, which contradicts $B$ being a brace.

    For the other direction, suppose \Cref{con:main} holds, but there exists a 2-factor Hamiltonian, cubic, bipartite graph $G$ that cannot be built from $K_{3,3}$ and the Heawood graph via the star product.
    We choose $G$ such that it is a smallest counterexample to \Cref{con:2FH}.
    First we note that according to \Cref{lem:3con} the graph $G$ must be 3-connected and $G$ cannot itself be a brace, as otherwise  \Cref{con:main} implies that $G$ has to be $K_{3,3}$ or the Heawood graph.
    Thus $G$ contains a non-trivial tight cut $\Cut{X}$.
    According to \Cref{lem:tightcutpreserve}, both $G_1$ and $G_2$ are 3-connected, cubic, and bipartite.
    Furthermore, since $G$ is 2-factor Hamiltonian, \Cref{lem:tightcut2FH} tells us that both $G_1$ and $G_2$ are also 2-factor Hamiltonian.
	Due to the minimality of $G$, it must therefore be possible to construct both $G_1$ and $G_2$ via the star product from $K_{3,3}$ or the Heawood graph.
	However, \Cref{obs:starproductcorrespondence} now implies that $G$ itself can be constructed from $K_{3,3}$ or the Heawood graph, a contradiction.
\end{proof}

\section{Cycles of length four in Pfaffian braces}\label{sec:pfaffian}

Our goal now is to show that all Pfaffian braces except the Heawood graph have girth 4, as stated in \Cref{thm:main}.
We start our journey with planar braces and show that they contain many 4-cycles.
Given a 3-connected, planar graph $G$, we associate $G$ with its drawing in the plane, which is known to be essentially unique thanks to a result by Whitney \cite{Whitney19332Isomorphic}, and let $F(G)$ be the faces of said drawing.
Since braces other than $C_4$ are 3-connected according to \Cref{thm:extendabilitybasics}, this suffices for our efforts.
According to another result by Whitney all faces of braces thus correspond to cycles \cite{Whitney1932Nonseparable}.
We will also use the classic formula $|V(G)| - |E(G)| + |F(G)| = 2$ by Euler for planar graphs, which can be derived from \cite{Euler1758Elementa}.

\begin{lemma}\label{lem:planar4cycles}
    All planar braces $B$ other than $C_4$ contain at least six distinct 4-cycles each corresponding to a face of $B$.
\end{lemma}
\begin{proof}
    Let $B$ be a planar brace other than $C_4$.
    As just noted, $B$ is 2-extendable and thus 3-connected according to \Cref{thm:extendabilitybasics}.
    Therefore $B$ has minimum degree at least 3 and we have
    \begin{align}
        \nonumber & \ 2|E(B)| = \sum_{u \in V(B)} deg(u) \geq \sum_{u \in V(B)} 3 = 3|V(B)| \\
        \Leftrightarrow & \ \nicefrac{2}{3}|E(B)| \geq |V(B)| . \label[Ineq]{eq:edgesVSvertices}
    \end{align}
    We let $c_4$ be the number of faces corresponding to 4-cycles in $B$ and for any $f \in F(B)$, we let $d(f)$ be the length of the cycle corresponding to $f$ in $B$.
    As $B$ is bipartite, any cycle in $B$ must have an even number of vertices and thus $d(f)$ is even and at least 4 for all $f \in F(B)$.
    In particular, we note that $\sum_{f \in F(B)} d(f) = 2|E(B)|$.
    This leads to
    \begin{align}
        \nonumber & \ 2|E(B)| = \sum_{f \in F(B)} d(f) \geq 6(F(B) - c_4) + 4c_4 = 6|F(B)| - 2c_4 \\
        \nonumber \Leftrightarrow & \ 2|E(B)| + 2c_4 \geq 6 |F(B)| \\
        \Leftrightarrow & \ \nicefrac{1}{3}|E(B)| + \nicefrac{c_4}{3} \geq |F(B)| . \label[Ineq]{eq:edgesVSfaces}
    \end{align}
    By inserting (\ref{eq:edgesVSvertices}) and (\ref{eq:edgesVSfaces}) into Euler's formula, we get
    \[ 2 = |V(B)| - |E(B)| + |F(B)| \leq \nicefrac{2}{3} |E(B)| - |E(B)| + \nicefrac{1}{3} |E(B)| + \nicefrac{c_4}{3} = \nicefrac{c_4}{3} . \]
    This yields $c_4 \geq 6$ and thus $B$ contains at least six distinct 4-cycles.
\end{proof}

\Cref{lem:planar4cycles} is tight, as can be seen in the cube (see \Cref{fig:heawood}).
We will also need a useful result on separating 4-cycles in braces from McCuaig.

\begin{lemma}[McCuaig \cite{McCuaig2004Polyas}]\label{lem:4cyclesplit}
    Let $G_1$ and $G_2$ be bipartite such that $G_1 \cap G_2$ is a 4-cycle.
    If $G_1 \cup G_2$ is a brace then $G_1$ and $G_2$ are braces.
\end{lemma}

Using this lemma we can now prove that we can find edge-disjoint 4-cycles within planar braces, where two cycles $C,C'$ are called \emph{edge-disjoint} if $E(C) \cap E(C') = \emptyset$.

\begin{lemma}\label{lem:edgedisjoint4cycles}
    Let $B$ be a planar brace other than $C_4$.
    For each 4-cycle $C$ in $B$ there exists a 4-cycle $C'$ in $B$ that is edge-disjoint from $C$.
\end{lemma}
\begin{proof}
    Let $C$ be a 4-cycle in $B$.
    Suppose $C$ corresponds to a face of $B$.
    According to \Cref{lem:planar4cycles} there exists at least five more faces bounded by 4-cycles.
    If none of these 4-cycles are edge-disjoint from $C$, then according to the pigeon hole principle there exists an edge $e \in E(C)$ that is found in two of these cycles.
    However, this implies that $e$ is contained in the boundary of three faces, which is impossible.
    Thus there must exist a 4-cycle $C'$ that is edge-disjoint from $C$ in this case.

    Thus, we can instead suppose $C$ does not correspond to a face of $B$.
    Accordingly $C$ separates the graph and in particular there exist two subgraphs $G_1, G_2 \subseteq G$ with $G_1 \cap G_2 = C$ and $G_1 \cup G_2 = G$.
    This allows us to apply \Cref{lem:4cyclesplit} and conclude that both $G_1$ and $G_2$ are braces.
    Since $G$ is planar, both $G_1$ and $G_2$ must also be planar and in particular, $C$ bounds a face in both of these graphs.
    As discussed in the previous paragraph, this means that both $G_1$ and $G_2$ contain a 4-cycle that is edge-disjoint from $C$ and thus $G$ itself also contains such a 4-cycle.
\end{proof}

This then allows us to prove the existence of edge-disjoint 4-cycles in non-planar braces.

\begin{lemma}\label{lem:nonplanar4cycles}
    Let $B$ be a non-planar Pfaffian brace other than the Heawood graph.
    Then $B$ contains three edge-disjoint 4-cycles and for each 4-cycle $C$ in $B$ there exists a 4-cycle $C'$ in $B$ that is edge-disjoint from $C$.
\end{lemma}
\begin{proof}
    We prove this statement by induction.
    Let $B$ be a non-planar, Pfaffian brace such that for all non-planar, Pfaffian braces with less vertices the statement holds.

    According to \Cref{thm:pfaffian} the brace $B$ is the trisum of three braces $B_1, B_2, B_3$, each having less vertices than $B$.
    Note that each of these braces may be planar.
    Let $C$ be the 4-cycle that forms the pairwise intersection of $B_1$, $B_2$, and $B_3$.
    By induction hypothesis and \Cref{lem:edgedisjoint4cycles} for each $i \in [3]$ there exists a 4-cycle $C^i \subseteq B_i$ that is edge-disjoint from $C$.

    Now let $H$ be a 4-cycle in $B$.
    If $H \subseteq B_i$ for some $i \in [3]$, then clearly both $C^j$ and $C^h$ with $j,h \in [3] \setminus \{ i \}$ are edge-disjoint from $H$.
    Accordingly $V(H) \cap V(C) \neq \emptyset$ and in particular $|V(H) \cap V(C)| \geq 2$, as otherwise $H$ would be contained in $B_i$ for some $i \in [3]$.
    Furthermore, if $|V(H) \cap V(C)| \geq 3$, we again have $H \subseteq B_i$ for some $i \in [3]$ and thus we conclude that $|V(H) \cap V(C)| = 2$.
    But this means that there must exist two distinct $i,j \in [3]$ such that $H \subseteq B_i \cup B_j$ and thus $C^h$ with $h \in [3] \setminus \{ i, j \}$ is edge-disjoint from $H$, completing our proof.
\end{proof}

\Cref{thm:main} is then an easy consequence of combining \Cref{lem:planar4cycles} and \Cref{lem:nonplanar4cycles}.

\bigskip

\textbf{Acknowledgements:}
The authors would like to thank an anonymous reviewer for several helpful comments that improved the presentation of the article.

This article is based on the Bachelor's thesis of the second author \cite{Johanni20242Factor}, which was supervised by the first author.

\bibliographystyle{alphaurl}
\bibliography{literature}

\appendix

\section{Some notes on the star product and 2-factor Hamiltonicity}\label{app:starproduct}

We start by noting that it is somewhat of a misnomer to speak of \emph{the} star product of two graphs, as the particular graph that results from this operation depends on which edges are added as can be seen in \Cref{fig:nonbip}.
It is also notable that the definition of 2-factor Hamiltonian graphs does not require the graph under consideration to actually have a 2-factor.
Thus a graph without any 2-factors is 2-factor Hamiltonian (see \Cref{fig:2fhstrange}).

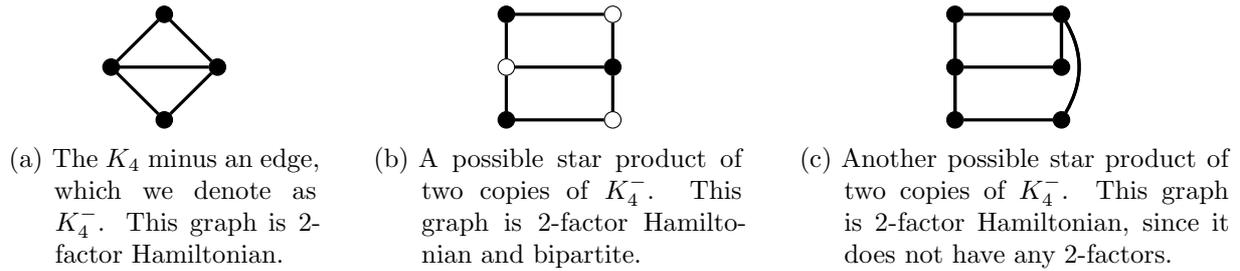
\begin{figure}[ht] 
\centering
\begin{subfigure}{0.25\textwidth}
    \centering
        \begin{tikzpicture}[scale=0.7]
            \node (A1) at (0,1) [draw, circle, fill, scale=0.6] {};
            \node (A2) at (1,0) [draw, circle, fill, scale=0.6] {};
            \node (A3) at (2,1) [draw, circle, fill, scale=0.6] {};
            \node (A4) at (1,2) [draw, circle, fill, scale=0.6] {};
            
            \foreach\i in {1,2,3}
            {
                \pgfmathtruncatemacro\iplus{\i+1}
                \path (A\i) edge[very thick] (A\iplus);
            }
            \path[very thick]
                (A3) edge (A1)
                (A4) edge (A1)
            ;
		\end{tikzpicture}
    \caption{The $K_4$ minus an edge, which we denote as $K_4^-$.
    This graph is 2-factor Hamiltonian.}
\end{subfigure}
\hfill
\begin{subfigure}{0.3\textwidth}
    \centering
    	\begin{tikzpicture}[scale=0.7]
            \node (A1) at (0,0) [draw, circle, fill, scale=0.6] {};
            \node (A2) at (2,1) [draw, circle, fill, scale=0.6] {};
            \node (A3) at (0,2) [draw, circle, fill, scale=0.6] {};
            
            \node (B1) at (2,0) [draw, circle, scale=0.6] {};
            \node (B2) at (0,1) [draw, circle, scale=0.6] {};
            \node (B3) at (2,2) [draw, circle, scale=0.6] {};
            
            \foreach\i in {1,2,3}
            {
                \path (A\i) edge[very thick] (B\i);
            }
            \foreach\i in {1,2}
            {
                \pgfmathtruncatemacro\iplus{\i+1}
                \path
                    (A\i) edge[very thick] (B\iplus)
                    (B\i) edge[very thick] (A\iplus)
                ;
            }
		\end{tikzpicture}
    \caption{A possible star product of two copies of $K_4^-$.
    This graph is 2-factor Hamiltonian and bipartite.}
    \label{fig:bipfromnonbip}
\end{subfigure}
\hfill
\begin{subfigure}{0.35\textwidth}
    \centering
		\begin{tikzpicture}[scale=0.7]

            \node (A1) at (0,0) [draw, circle, fill, scale=0.6] {};
            \node (A2) at (2,1) [draw, circle, fill, scale=0.6] {};
            \node (A3) at (0,2) [draw, circle, fill, scale=0.6] {};
            
            \node (B1) at (2,0) [draw, circle, fill, scale=0.6] {};
            \node (B2) at (0,1) [draw, circle, fill, scale=0.6] {};
            \node (B3) at (2,2) [draw, circle, fill, scale=0.6] {};
            
            \foreach\i in {1,2,3}
            {
                \path (A\i) edge[very thick] (B\i);
            }
            \foreach\i in {1,2}
            \path[very thick]
                (A1) edge (B2)
                (B2) edge (A3)
                (A2) edge (B3)
                (B1) edge[bend right] (B3)
            ;
            
		\end{tikzpicture}
    \caption{Another possible star product of two copies of $K_4^-$.
    This graph is 2-factor Hamiltonian, since it does not have any 2-factors.}
    \label{fig:2fhstrange}
\end{subfigure}
    \caption{Two different star products of the same simple graph that exhibit unexpected behaviour.}
    \label{fig:nonbip}
\end{figure}

Furthermore, somewhat unexpectedly, it is possible to build bipartite graphs as the star product of two non-bipartite graphs (see \Cref{fig:bipfromnonbip}).
For this to happen, the two graphs in the star product must have the property that the two vertices used to perform the operation are contained in all odd cycles of their respective graph.

To see why these remarks are relevant, we discuss the following proposition in \cite{FunkJLS20032Factor} that is presented without proof:
If $G$ is a bipartite graph that can be represented as the star product $(G_1,x) * (G_2,y)$, then $G$ is 2-factor Hamiltonian if and only if $G_1$ and $G_2$ are 2-factor Hamiltonian.

\begin{figure}[ht] 
\centering
\begin{subfigure}{0.4\textwidth}
    \centering
        \begin{tikzpicture}[scale=0.7]

            \node (A1) at (0,2) [draw, circle, fill, scale=0.6] {};
            \node (A2) at (2,1) [draw, circle, fill, scale=0.6] {};
            \node (A3) at (2,3) [draw, circle, fill, scale=0.6] {};
            \node (A4) at (4,2) [draw, circle, fill, scale=0.6] {};
            
            \node (B1) at (2,0) [draw, circle, scale=0.6] {};
            \node (B2) at (2,2) [draw, circle, scale=0.6] {};
            \node (B3) at (2,4) [draw, circle, scale=0.6] {};
            
            \path[very thick]
                (A1) edge (B1)
                (A1) edge (B2)
                (A1) edge (B3)
                (A4) edge (B1)
                (A4) edge (B2)
                (A4) edge (B3)
                (A2) edge (B1)
                (A2) edge (B2)
                (A3) edge (B2)
                (A3) edge (B3)
                ;
		\end{tikzpicture}
    \caption{A bipartite graph that is 2-factor Hamiltonian, due to not containing a 2-factor.}
\end{subfigure}
\qquad \qquad
\begin{subfigure}{0.4\textwidth}
    \centering
    	\begin{tikzpicture}[scale=0.7]

            \node (A1) at (0,2) [draw, circle, fill, scale=0.6] {};
            \node (A2) at (2,1) [draw, circle, fill, scale=0.6] {};
            \node (A3) at (2,3) [draw, circle, fill, scale=0.6] {};
            \node (A4) at (4,0) [draw, circle, fill, scale=0.6] {};
            \node (A5) at (4,2) [draw, circle, fill, scale=0.6] {};
            \node (A6) at (4,4) [draw, circle, fill, scale=0.6] {};
            
            \node (B1) at (2,0) [draw, circle, scale=0.6] {};
            \node (B2) at (2,2) [draw, circle, scale=0.6] {};
            \node (B3) at (2,4) [draw, circle, scale=0.6] {};
            \node (B4) at (4,1) [draw, circle, scale=0.6] {};
            \node (B5) at (4,3) [draw, circle, scale=0.6] {};
            \node (B6) at (6,2) [draw, circle, scale=0.6] {};
            
            \path[very thick]
                (A1) edge (B1)
                (A1) edge (B2)
                (A1) edge (B3)
                (A2) edge (B1)
                (A2) edge (B2)
                (A3) edge (B2)
                (A3) edge (B3)
                (A4) edge (B1)
                (A4) edge (B4)
                (A4) edge (B6)
                (A5) edge (B2)
                (A5) edge (B4)
                (A5) edge (B5)
                (A5) edge (B6)
                (A6) edge (B3)
                (A6) edge (B5)
                (A6) edge (B6)
                ;
		\end{tikzpicture}
    \caption{A bipartite graph that is not 2-factor Hamiltonian.}
\end{subfigure}
    \caption{The graph on the right is the star product of two copies of the graph to the left.}
    \label{fig:bipstrange}
\end{figure}
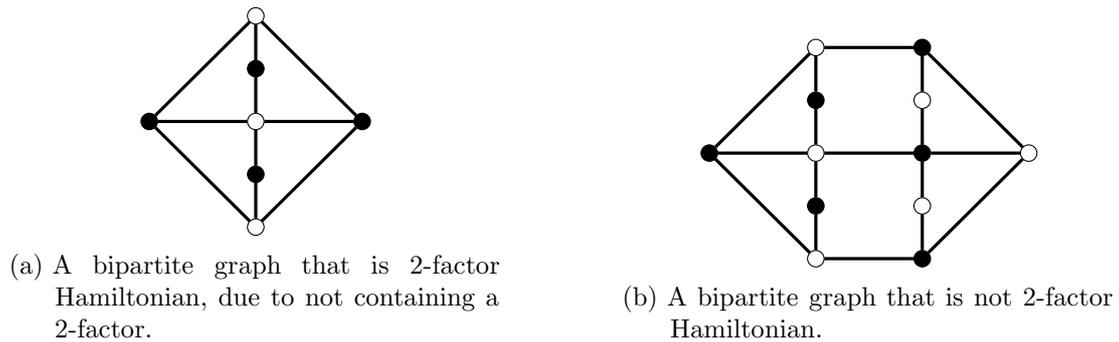
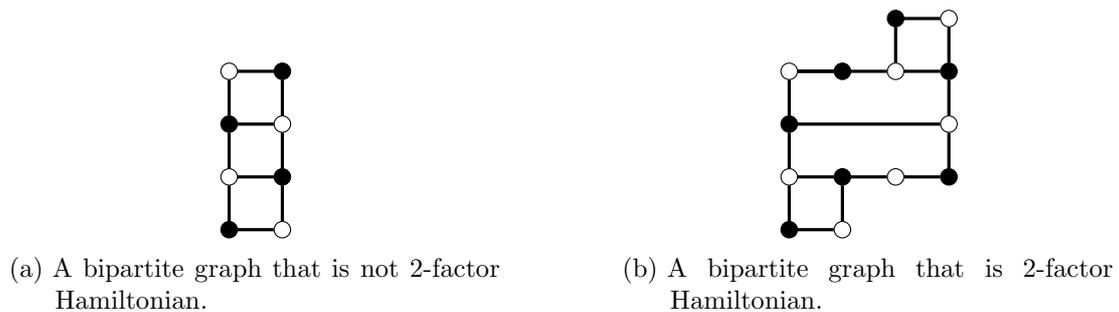

Both directions of this statement fail, as can be seen in \Cref{fig:bipstrange} and \Cref{fig:bipstrange2}.
In particular, even if we adjust the definition of 2-factor Hamiltonicity to demand that the graph in question must have a 2-factor, the forward direction still fails (see \Cref{fig:bipstrange2}).
This proposition has since found its way into several other articles (see \cite{Labbate2004Amalgams,FouquetTV2010Family,RomanielloZ2023Betwixt,MatsumotoNY2024Cubic}).
Luckily, as far as we can tell, this has not caused any further erroneous statements to be published.

\begin{figure}[ht] 
\centering
\begin{subfigure}{0.4\textwidth}
    \centering
		\begin{tikzpicture}[scale=0.7]

            \node (A1) at (0,0) [draw, circle, fill, scale=0.6] {};
            \node (A2) at (1,1) [draw, circle, fill, scale=0.6] {};
            \node (A3) at (0,2) [draw, circle, fill, scale=0.6] {};
            \node (A4) at (1,3) [draw, circle, fill, scale=0.6] {};
            
            \node (B1) at (1,0) [draw, circle, scale=0.6] {};
            \node (B2) at (0,1) [draw, circle, scale=0.6] {};
            \node (B3) at (1,2) [draw, circle, scale=0.6] {};
            \node (B4) at (0,3) [draw, circle, scale=0.6] {};
            
            \foreach\i in {1,2,3,4}
            {
                \path (A\i) edge[very thick] (B\i);
            }
            \foreach\i in {1,2,3}
            {
                \pgfmathtruncatemacro\iplus{\i+1}
                \path
                    (A\i) edge[very thick] (B\iplus)
                    (B\i) edge[very thick] (A\iplus)
                ;
            }
		\end{tikzpicture}
    \caption{A bipartite graph that is not 2-factor Hamiltonian.}
\end{subfigure}
\qquad \qquad
\begin{subfigure}{0.4\textwidth}
    \centering
		\begin{tikzpicture}[scale=0.7]

            \node (A1) at (0,0) [draw, circle, fill, scale=0.6] {};
            \node (A2) at (1,1) [draw, circle, fill, scale=0.6] {};
            \node (A3) at (0,2) [draw, circle, fill, scale=0.6] {};
            \node (A4) at (1,3) [draw, circle, fill, scale=0.6] {};
            
            \node (B1) at (1,0) [draw, circle, scale=0.6] {};
            \node (B2) at (0,1) [draw, circle, scale=0.6] {};
            \node (B3) at (0,3) [draw, circle, scale=0.6] {};
            
            \node (C1) at (2,1) [draw, circle, scale=0.6] {};
            \node (C2) at (3,2) [draw, circle, scale=0.6] {};
            \node (C3) at (2,3) [draw, circle, scale=0.6] {};
            \node (C4) at (3,4) [draw, circle, scale=0.6] {};

            \node (D1) at (3,1) [draw, circle, fill, scale=0.6] {};
            \node (D2) at (3,3) [draw, circle, fill, scale=0.6] {};
            \node (D3) at (2,4) [draw, circle, fill, scale=0.6] {};
            
            \foreach\i in {1,2,3}
            {
                \path (A\i) edge[very thick] (B\i);
            }
            \foreach\i in {1,2,3}
            {
                \path (C\i) edge[very thick] (D\i);
            }

            \path[very thick]
                (A1) edge (B2)
                (A2) edge (B1)
                (A2) edge (C1)
                (A3) edge (B2)
                (A3) edge (C2)
                (A4) edge (B3)
                (B3) edge (C3)
                (D1) edge (C2)
                (C3) edge (D2)
                (C4) edge (D2)
                (C4) edge (D3)
            ;
		\end{tikzpicture}
    \caption{A bipartite graph that is 2-factor Hamiltonian.}
\end{subfigure}
    \caption{The graph on the right is the star product of two copies of the graph to the left.}
    \label{fig:bipstrange2}
\end{figure}

In \cite{FunkJLS20032Factor} itself it turns out that in the proof of their main results this proposition is only needed in a very limited setting and there it actually holds.
In particular, Funk et al. only need their proposed statement to work if $G$, $G_1$, and $G_2$ are all bipartite and cubic.
Thus our goal is to prove the following.

\begin{lemma}\label{lem:starproductfixbip}
    Let $G$ be a cubic, bipartite graph that can be represented as the star product $(G_1,x) * (G_2,y)$ of two cubic, bipartite graphs $G_1$ and $G_2$.
    Then $G$ is 2-factor Hamiltonian if and only if $G_1$ and $G_2$ are 2-factor Hamiltonian\footnote{In \cite{Diwan2002Disconnected} Diwan already points out that the reverse direction of the statement holds, again without proof.}.
\end{lemma}

We now present and prove two lemmas which lead towards a proof of \Cref{lem:starproductfixbip}.
The statement can be proven via easier methods, closely related to what we present in \Cref{sec:braces}, but this route allows us to provide a proof for another claim concerning the star product from the literature and also relate another generalisation of \Cref{lem:starproductfixbip} to the reader that came up during discussions with an author of \cite{FunkJLS20032Factor}.

Our first avenue towards a proof of \Cref{lem:starproductfixbip} is taken from \cite[p.\ 1851]{AbreuLS2012Pseudo}.
To state their proposition, which they provide without a proof, we need to define a few more concepts.
First, we call a cut $\Cut{X}$ in a graph \emph{quasi-tight} if $|\Cut{X} \cap M| = 1$ for every perfect matching $M$ of $G$.
We distinguish this notion from tight cuts, which have the same requirement, since we only defined tight cuts for matching covered graphs and cuts are always quasi-tight in graphs that do not have perfect matchings.
Furthermore, a \emph{bridge} in a graph $G$ is an edge $e$ such that $G - e$ is not connected.
We call a graph $G$ \emph{bridgeless} if it does not contain a bridge.

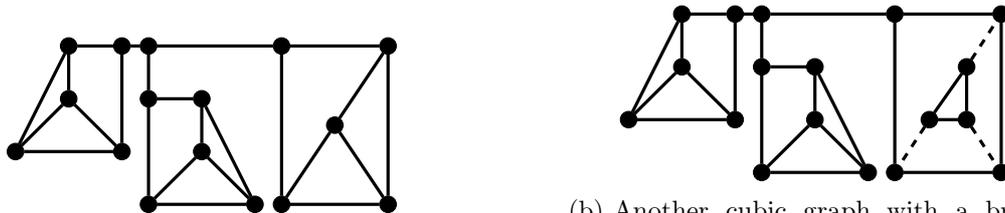
\begin{figure}[ht] 
\centering
\begin{subfigure}{0.4\textwidth}
    \centering
		\begin{tikzpicture}[scale=0.7]
            \node (A1) at (0,1) [draw, circle, fill, scale=0.6] {};
            \node (A2) at (1,3) [draw, circle, fill, scale=0.6] {};
            \node (A3) at (1,2) [draw, circle, fill, scale=0.6] {};
            \node (A4) at (2,3) [draw, circle, fill, scale=0.6] {};
            \node (A5) at (2,1) [draw, circle, fill, scale=0.6] {};
            \node (A6) at (2.5,3) [draw, circle, fill, scale=0.6] {};
            \node (A7) at (2.5,2) [draw, circle, fill, scale=0.6] {};
            \node (A8) at (2.5,0) [draw, circle, fill, scale=0.6] {};
            \node (A9) at (3.5,2) [draw, circle, fill, scale=0.6] {};
            \node (A10) at (3.5,1) [draw, circle, fill, scale=0.6] {};
            \node (A11) at (4.5,0) [draw, circle, fill, scale=0.6] {};
            \node (A12) at (5,3) [draw, circle, fill, scale=0.6] {};
            \node (A13) at (5,0) [draw, circle, fill, scale=0.6] {};
            \node (A14) at (6,1.5) [draw, circle, fill, scale=0.6] {};
            \node (A15) at (7,3) [draw, circle, fill, scale=0.6] {};
            \node (A16) at (7,0) [draw, circle, fill, scale=0.6] {};

            \path[very thick]
                (A1) edge (A2)
                (A1) edge (A3)
                (A1) edge (A5)
                (A2) edge (A3)
                (A2) edge (A4)
                (A3) edge (A5)
                (A4) edge (A5)
                (A4) edge (A6)
                (A6) edge (A7)
                (A6) edge (A12)
                (A7) edge (A8)
                (A7) edge (A9)
                (A8) edge (A10)
                (A8) edge (A11)
                (A9) edge (A10)
                (A9) edge (A11)
                (A10) edge (A11)
                (A12) edge (A13)
                (A12) edge (A15)
                (A13) edge (A14)
                (A13) edge (A16)
                (A14) edge (A15)
                (A14) edge (A16)
                (A15) edge (A16)
            ;
			\end{tikzpicture}
    \caption{A cubic graph with a bridge but without a perfect matching or a 2-factor.}
\end{subfigure}
\qquad \qquad
\begin{subfigure}{0.4\textwidth}
    \centering
		\begin{tikzpicture}[scale=0.7]
            \node (A1) at (0,1) [draw, circle, fill, scale=0.6] {};
            \node (A2) at (1,3) [draw, circle, fill, scale=0.6] {};
            \node (A3) at (1,2) [draw, circle, fill, scale=0.6] {};
            \node (A4) at (2,3) [draw, circle, fill, scale=0.6] {};
            \node (A5) at (2,1) [draw, circle, fill, scale=0.6] {};
            \node (A6) at (2.5,3) [draw, circle, fill, scale=0.6] {};
            \node (A7) at (2.5,2) [draw, circle, fill, scale=0.6] {};
            \node (A8) at (2.5,0) [draw, circle, fill, scale=0.6] {};
            \node (A9) at (3.5,2) [draw, circle, fill, scale=0.6] {};
            \node (A10) at (3.5,1) [draw, circle, fill, scale=0.6] {};
            \node (A11) at (4.5,0) [draw, circle, fill, scale=0.6] {};
            \node (A12) at (5,3) [draw, circle, fill, scale=0.6] {};
            \node (A13) at (5,0) [draw, circle, fill, scale=0.6] {};
            \node (A15) at (7,3) [draw, circle, fill, scale=0.6] {};
            \node (A16) at (7,0) [draw, circle, fill, scale=0.6] {};

            \node (C1) at (5.65,1) [draw, circle, fill, scale=0.6] {};
            \node (C2) at (6.35,1) [draw, circle, fill, scale=0.6] {};
            \node (C3) at (6.35,2) [draw, circle, fill, scale=0.6] {};

            \path[very thick]
                (C1) edge (C2)
                (C1) edge (C3)
                (C1) edge[dashed] (A13)
                (C2) edge (C3)
                (C2) edge[dashed] (A16)
                (C3) edge[dashed] (A15)
                (A1) edge (A2)
                (A1) edge (A3)
                (A1) edge (A5)
                (A2) edge (A3)
                (A2) edge (A4)
                (A3) edge (A5)
                (A4) edge (A5)
                (A4) edge (A6)
                (A6) edge (A7)
                (A6) edge (A12)
                (A7) edge (A8)
                (A7) edge (A9)
                (A8) edge (A10)
                (A8) edge (A11)
                (A9) edge (A10)
                (A9) edge (A11)
                (A10) edge (A11)
                (A12) edge (A13)
                (A12) edge (A15)
                (A13) edge (A16)
                (A15) edge (A16)
            ;
			\end{tikzpicture}
    \caption{Another cubic graph with a bridge that does not have a perfect matching or a 2-factor.}
\end{subfigure}
    \caption{The graph to the right is the star product of the graph to the left and $K_4$.
    Since the graph to the right does not have a perfect matching, the principal 3-edge cut within it, marked with dashed lines, must be quasi-tight.}
    \label{fig:bridgeexample}
\end{figure}

\begin{lemma}[Abreu, Labbate, and Sheehan \cite{AbreuLS2012Pseudo}]\label{lem:starproductfix1}
    Let $G = (G_1,v_1) * (G_2,v_2)$ be a cubic, bridgeless graph.
    Then $G$ is 2-factor Hamiltonian if and only if both $G_1$ and $G_2$ are 2-factor Hamiltonian and the principal 3-edge cut in $G$ is quasi-tight.
\end{lemma}

We define all of the notions above because in \cite{AbreuLS2012Pseudo} \Cref{lem:starproductfix1} is stated without asking for $G$ to be bridgeless.
This causes the lemma to fail as can be seen in \Cref{fig:bridgeexample}.
However, this is clearly just a minor oversight stemming from the likely unintended gaps in the definitions of 2-factor Hamiltonicity and tight cuts, as they are defined in \cite{AbreuLS2012Pseudo}.
Asking for $G$ to have a 2-factor, or equivalently a perfect matching, easily fixes this issue as well.
We thus just make this remark here for the sake of completeness.

As pointed out in \cite{AbreuLS2012Pseudo}, the last part of \Cref{lem:starproductfix1} concerning the principal 3-edge cut in $G$ having to be quasi-tight is necessary as can be seen in \Cref{fig:triangularprism}.
Since \Cref{lem:starproductfix1} is stated without a proof in \cite{AbreuLS2012Pseudo}, we want to show that this lemma holds.
For this purpose, we will need more information on the conditions under which a cubic graph is matching covered, as we want to make use of the following observation.

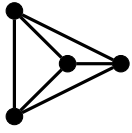
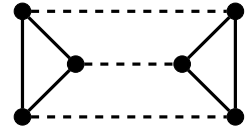
\begin{figure}[ht] 
\centering
\begin{subfigure}{0.4\textwidth}
    \centering
		\begin{tikzpicture}[scale=0.7]
            \node (A1) at (0,0) [draw, circle, fill, scale=0.6] {};
            \node (A2) at (0,2) [draw, circle, fill, scale=0.6] {};
            \node (A3) at (1,1) [draw, circle, fill, scale=0.6] {};
            \node (A4) at (2,1) [draw, circle, fill, scale=0.6] {};
            
            \foreach\i in {1,2,3}
            {
                \pgfmathtruncatemacro\iplus{\i+1}
                \path (A\i) edge[very thick] (A\iplus);
            }
            \path[very thick]
                (A2) edge (A4)
                (A4) edge (A1)
                (A1) edge (A3)
            ;
		\end{tikzpicture}
    \caption{The $K_4$, which is 2-factor Hamiltonian.}
\end{subfigure}
\qquad \qquad
\begin{subfigure}{0.4\textwidth}
    \centering
		\begin{tikzpicture}[scale=0.7]
            \node (A1) at (0,0) [draw, circle, fill, scale=0.6] {};
            \node (A2) at (0,2) [draw, circle, fill, scale=0.6] {};
            \node (A3) at (1,1) [draw, circle, fill, scale=0.6] {};
            \node (A4) at (3,1) [draw, circle, fill, scale=0.6] {};
            \node (A5) at (4,0) [draw, circle, fill, scale=0.6] {};
            \node (A6) at (4,2) [draw, circle, fill, scale=0.6] {};
            
            \foreach\i in {1,2,4,5}
            {
                \pgfmathtruncatemacro\iplus{\i+1}
                \path (A\i) edge[very thick] (A\iplus);
            }
            \path[very thick]
                (A2) edge[dashed] (A6)
                (A4) edge (A6)
                (A1) edge[dashed] (A5)
                (A1) edge (A3)
                (A3) edge[dashed] (A4)
            ;
		\end{tikzpicture}
    \caption{The triangular prism, also denoted as $\Compl{C_6}$, which is not 2-factor Hamiltonian.}
\end{subfigure}
    \caption{The graph $\Compl{C_6}$ on the right is the star product of two copies of $K_4$ to the left.
    Notably the principal 3-edge cut in $\Compl{C_6}$, marked by the dashed lines, is not tight.}
    \label{fig:triangularprism}
\end{figure}

\begin{observation}\label{obs:perfmat}
    If $G$ is a 2-factor Hamiltonian, cubic graph with a perfect matching $M$, then $G - M$ is a Hamiltonian cycle of $G$.
\end{observation}

To get a characterisation of matching covered, cubic graphs, we make use of a result due to Tutte.

\begin{theorem}[Tutte \cite{Tutte1947Factorization} (see the Corollary on p.\ 111)]\label{cor:tutte}
For any integer $k \geq 2$, all $k$-regular, $k-1$-connected graphs with an even number of vertices are matching covered.
\end{theorem}

We note that a close inspection of an earlier article by Schönberger in German contains the arguments necessary to show that any cubic, 2-edge connected graph is matching covered \cite{Schoenberger1934Beweis}.
Regardless, using \Cref{cor:tutte}, we can prove the following neat characterisation.

\begin{theorem}\label{thm:matcov}
    A cubic graph is matching covered if and only if it is 2-connected.
\end{theorem}
\begin{proof}
    The statement clearly holds for $K_4$.
    All other cubic, matching covered graphs have six or more vertices and thus \Cref{thm:extendabilitybasics} tells us that they are 2-connected.
    If we have a cubic, 2-connected graph, we can in turn use \Cref{cor:tutte} to see that they are matching covered.
\end{proof}

We will also require the following classic result by Petersen that guarantees that we can find perfect matchings and thus 2-factors in our graphs.

\begin{theorem}[Petersen \cite{Petersen1891Theorie}]\label{thm:petersen}
    Every cubic, bridgeless graph has a perfect matching.
\end{theorem}

We are now ready to prove \Cref{lem:starproductfix1}.
Note that some of the arguments here, especially in the forward direction, resemble the arguments we presented in the proof of \Cref{lem:tightcut2FH}.

\begin{proof}[Proof of \Cref{lem:starproductfix1}]
    Let the principal 3-edge cut in $G$ be the set $\Cut{X} = \{ e,f,g \}$ with $X = V(G_1) \setminus \{ v_1 \}$.
    For both $i \in [2]$, we let $e_i$ be the unique edge incident to $v_i$ in $G_i$ that shares an endpoint with $e$ and we define $f_i$ and $g_i$ analogously.
    
    Suppose that $G_1$ is not bridgeless.
    Clearly $G_1$ cannot have more than one connected component, as this would imply that $G$ itself is not connected, which contradicts $G$ being bridgeless.
    Let $b$ be a bridge of $G_1$.
    Then we must have $b \not\in E(G)$, as otherwise $G$ would not be bridgeless.
    Thus $b$ must be incident to $v_1$ and without loss of generality we assume that $b = e_1$.
    But this implies that $e$ is also a bridge, another contradiction.
    Thus $G_1$ and analogously $G_2$ are bridgeless.

    Next we note that \Cref{thm:petersen} tells us that any bridgeless graph contains a 2-factor and if said graph is also 2-factor Hamiltonian this implies that it is 2-connected and thus matching covered according to \Cref{thm:matcov}.

    Suppose now that $G$ is 2-factor Hamiltonian, but at least one of the graphs $G_1$ and $G_2$ is not.
    Without loss of generality, we assume that there exists a 2-factor $F$ in $G_1$ that is not a Hamiltonian cycle.
    We further assume without loss of generality that $e_1 \not\in E(F)$.
    Then we note that since $G$ is matching covered, $G$ has a perfect matching containing $e$.
    Therefore, according to \Cref{obs:perfmat}, the graph $G$ contains a 2-factor $H$ that does not contain $e$, which is a Hamiltonian cycle in $G$, since $G$ is 2-factor Hamiltonian.
    Clearly, $(F - v_1) \cup (H - V(G_1)) \cup \{ f,g \}$ is a 2-factor of $G$ that is not a Hamiltonian cycle, contradicting $G$ being 2-factor Hamiltonian.
    Thus, as analogous arguments work for $G_2$, both $G_1$ and $G_2$ must be 2-factor Hamiltonian.

    Still under the assumption that $G$ is 2-factor Hamiltonian, we must now further prove that $\Cut{X}$ is a quasi-tight cut in $G$.
    First, let $x \in \Cut{X}$.
    Since both $G_1$ and $G_2$ are matching covered, there exist two perfect matchings $M_1 \in E(G_1)$ and $M_2 \in E(G_2)$ such that $x_i \in M_i$ for both $i \in [2]$.
    Clearly, $M = (M_1 \cup M_2 \cup \{ x \}) \setminus \{ x_1, x_2 \}$ is a perfect matching of $G$ with $|M \cap \Cut{X}| = 1$.
    In particular, the existence of $M$ confirms that $|V(G_i) \setminus \{ v_i \}|$ has odd parity for both $i \in [2]$ and thus $|N \cap \Cut{X}|$ must also be odd for any perfect matching $N$ of $G$.
    As a consequence, if $\Cut{X}$ is not quasi-tight, there exists a perfect matching $M'$ of $G$ with $\Cut{X} \subseteq M'$.
    However, as noted in \Cref{obs:perfmat}, the graph $G - M'$ must be a 2-factor and since $\Cut{X}$ is a cut in $G$, this 2-factor cannot be a Hamiltonian cycle, which contradicts $G$ being 2-factor Hamiltonian.
    Thus the forward direction of the statement holds.

    Suppose instead that $G_1$ and $G_2$ are 2-factor Hamiltonian and that $\Cut{X}$ is a quasi-tight cut in $G$.
    If $G$ is not 2-factor Hamiltonian, then it does at least contain a 2-factor $F$ that is not a Hamiltonian cycle due to it being bridgeless, as noted earlier.
    In particular, since $\Cut{X}$ is quasi-tight and $E(G) \setminus E(F)$ is a perfect matching, $F$ uses exactly 2 edges of $\Cut{X}$.
    Without loss of generality we assume that $e,f \in E(F)$.
    Furthermore, we know that $F$ contains at least one cycle that is entirely contained either in $G_1$ or $G_2$, leading us to assume without loss of generality that this holds for $G_1$.
    We now note that $( (F - V(G_2)) \cup \{ v_1 \} ) \cup \{ e_1,f_1 \}$ is a non-Hamiltonian 2-factor of $G_1$, since the cycle $C \subseteq F$ with $e,f \in E(C)$ contains at least two vertices from $V(G_1)$, as a principal 3-edge cut is a matching.
    This contradicts the 2-factor Hamiltonicity of $G_1$ and we are therefore done.
\end{proof}

With this, we are almost ready to prove \Cref{lem:starproductfixbip}.
However, we require a technical statement on the behaviour of certain odd cuts in bipartite graphs.
Let $G$ be a bipartite graph and let $A,B \subseteq V(G)$ be a pair of disjoint sets with $A \cup B = V(G)$ such that all edges of $G$ have one endpoint in $A$ and the other in $B$, which must exists since $G$ is bipartite.
We call $A$ and $B$ the two \emph{colour classes} of $G$ and note that the choice of $A$ and $B$ may not be unique, which is something we will ignore, since it is not relevant to our arguments.

\begin{lemma}\label{lem:biptechnical}
    Let $G$ be a bipartite graph with colour classes $A$ and $B$, and let $\Cut{X}$ be a cut in $G$ such that all endpoints of the edges in $\Cut{X}$ that are contained in $X$ are found in $A$, and there exists a perfect matching $M$ of $G$ with $|\Cut{X} \cap M| = 1$.
    
    Then $|X \cap A| - |X \cap B| = 1 = |\Compl{X} \cap B| - |\Compl{X} \cap A|$ and for all perfect matchings $M'$ of $G$ we also have $|\Cut{X} \cap M'| = 1$.
\end{lemma}
\begin{proof}
    Since $|\Cut{X} \cap M| = 1$, we know that all but one of the vertices of $X$ is contained in an edge of $M$.
    For the edge $e \in \Cut{X} \cap M$ we know that the endpoint of $e$ in $X$ lies in $A$.
    Thus we must have $|X \cap A| - |X \cap B| = 1$ and $|\Compl{X} \cap B| - |\Compl{X} \cap A| = 1$ is an immediate consequence of that.
    For any perfect matching $M'$ we then note that $M'$ must cover all vertices of $X \cap B$, but it can only do so with vertices of $X \cap A$, since all endpoints of edges within $\Cut{X}$ that lie in $X$ also lie in $A$.
    The same of course holds for the vertices within $\Compl{X} \cap A$.
    Thus $M'$ contains exactly one edge of $\Cut{X}$.
\end{proof}

This leads us into our first proof of \Cref{lem:starproductfixbip}.

\begin{proof}[Proof of \Cref{lem:starproductfixbip} via \Cref{lem:starproductfix1}]
    Suppose that $G$ is 2-factor Hamiltonian in addition to being cubic and bipartite.
    \Cref{thm:konig} tells us that $G$ contains a 2-factor, implying that $G$ is Hamiltonian and thus 2-connected.
    In turns this allows us to apply \Cref{lem:starproductfix1} to prove that $G_1$ and $G_2$ are 2-factor Hamiltonian.

    For the other direction, suppose that $G_1$ and $G_2$ are 2-factor Hamiltonian in addition to being cubic and bipartite.
    We again note that $G_1$ and $G_2$ are matching covered.
    For both $i \in [2]$, let $A_i,B_i$ be the colour classes for $G_i$ and let these be chosen such that the neighbours of $v_1$ lie in $A_1$ and the neighbours of $v_2$ lie in $B_2$.
    Clearly $(A_1 \cup A_2) \setminus \{ v_2 \}$ and $(B_1 \cup B_2) \setminus \{ v_1 \}$ are valid colour classes for $G$.
    Furthermore, as argued in the proof of \Cref{lem:starproductfix1}, the fact that $G_1$ and $G_2$ are matching covered, allows us to show that there exist a perfect matching of $G$ that contains exactly one edge of the principal 3-edge cut $\Cut{X}$.
    Thus we satisfy the requirements of \Cref{lem:biptechnical} and know that for any perfect matching $M$ of $G$ we have $|\Cut{X} \cap M| = 1$.
    We can therefore apply \Cref{lem:starproductfix1} to confirm that $G$ is 2-factor Hamiltonian.
\end{proof}

As an alternative avenue towards a proof of \Cref{lem:starproductfixbip}, Jackson \cite{Jackson2024Personal} suggested the following.

\begin{lemma}\label{lem:starproductfix2}
    Let $G$ be a bipartite graph that can be represented as the star product $(G_1,v_1) * (G_2,v_2)$, and let any two edges of the principal 3-edge cut of $G$ occur together in some 2-factor of $G$.
    Then $G$ is 2-factor Hamiltonian if and only if $G_1$ and $G_2$ are 2-factor Hamiltonian.
\end{lemma}
\begin{proof}
    Let the principal 3-edge cut in $G$ be the set $\{ e,f,g \} = \Cut{X}$ with $X = V(G_1) \setminus \{ v_1 \}$.
    For both $i \in [2]$, we let $e_i$ be the unique edge incident to $v_i$ in $G_i$ that shares an endpoint with $e$ and we define $f_i$ and $g_i$ analogously.
    Note that the requirement on $\Cut{X}$ immediately implies that $G$ actually contains a 2-factor, which saves us a lot of hassle.

    Suppose that $G$ is 2-factor Hamiltonian and that at least one of the graphs $G_1$ and $G_2$ is not 2-factor Hamiltonian.
    Without loss of generality, let this graph be $G_1$, with $F$ being a 2-factor of $G_1$ that is not a Hamiltonian cycle.
    We assume that $e_1,f_1 \in E(F)$, again without loss of generality, and let $H$ be a Hamiltonian cycle in $G$ that uses $e$ and $f$.
    The graph $((H - V(G_1)) \cup (F - v_1)) \cup \{ e,f \}$ is now clearly a 2-factor of $G$ that is not a Hamiltonian cycle, contradicting our assumptions on $G$.

    We can thus move on to supposing that $G_1$ and $G_2$ are 2-factor Hamiltonian and $G$ is not.
    Before we proceed further, we first study the structure of $\Cut{X}$ in $G$.
    Let $A$ and $B$ be the two colour classes of $G$ and suppose that the endpoints of the edges in $\Cut{X}$ that are found in $X$ are not all found in $A$ or $B$.
    Without loss of generality, we assume that two of these endpoints, particularly those of $e$ and $f$, are found in $A$ and the endpoint in $X$ belonging to $g$ is found in $B$.
    
    Consider a 2-factor $F'$ of $G$ that contains $e$ and $f$ and let $C'$ be the unique component of $F'$ that uses edges from $\Cut{X}$.
    Thus we have $e,f \in E(C')$.
    Note that since $G$ is bipartite any 2-factor of $G$ is also bipartite and thus in particular, $C'$ has even length.
    Therefore, since both endpoints of $e$ and $f$ that are found in $X$ lie in $A$, the path $C' \setminus \Compl{X}$ has even length and contains an odd number of vertices.
    All other components of $F'$ are contained entirely in $G_1$ or $G_2$.
    This then implies that $X$ contains an odd number of vertices.

    However, there also exists another 2-factor $F''$ in $G$ that contains the edges $e$ and $g$ and we let $C''$ be the unique component of $F''$ that uses edges from $\Cut{X}$.
    Since $C''$ has even length and the endpoints of $e$ and $g$ from $X$ belong to opposite colour classes, the path $C'' \setminus \Compl{X}$ is of odd length and contains an even number of vertices.
    This in turn implies that $X$ actually contains an even number of vertices, contradicting our earlier observation.
    Thus we actually know that all endpoints of edges from $\Cut{X}$ that lie in $X$ must all come from either $A$ or $B$.
    Without loss of generality we will assume that they all lie in $A$.

    Let us now again consider a 2-factor $F$ of $G$ using the edges $e$ and $f$.
    We note that of course $g \not\in E(F)$ and all components of $F$ are cycles of even length, as observed earlier.
    Thus, by applying \Cref{thm:konig}, we arrive at the conclusion that the set $E(F)$ can be partitioned into two disjoint perfect matchings $N,N'$ of $G$.
    In particular, since the endpoints in $X$ of both $e$ and $f$ lie in $A$, we must have $e \in N$ and $f \in N'$, or vice versa.
    Thus $N$ is a perfect matching of $G$ with $|\Cut{X} \cap N| = 1$ and this allows us to apply \Cref{lem:biptechnical} to conclude that we have $|\Cut{X} \cap M| = 1$ for all perfect matchings in $G$.
    As an extension of the above argument, we also have $|\Cut{X} \cap E(F)| = 2$ for all 2-factors of $G$.
    
    Suppose that there exists some 2-factor $F$ of $G$ that is not a Hamiltonian cycle, then we now know that $|E(F) \cap \Cut{X}| = 2$.
    We assume without loss of generality that $\{ e,f \} = E(F) \cap \Cut{X}$ and let $C$ be the component in $F$ that contains $e$ and $f$.
    Furthermore, since $F$ is not a Hamiltonian cycle, $F$ contains at least one component that is entirely found in $G_1$ or $G_2$ and we assume without loss of generality that this holds for $G_1$.
    Now the graph $((F - V(G_2)) \cup \{ v_1 \}) \cup \{ e_1,f_1 \}$ is a 2-factor in $G_1$ that is not a Hamiltonian cycle and thus contradicts the 2-factor Hamiltonicity of $G_1$.
\end{proof}

It is now fairly easy to prove \Cref{lem:starproductfixbip} via this result and thus conclude this appendix.

\begin{proof}[Proof of \Cref{lem:starproductfixbip} via \Cref{lem:starproductfix2}]
    Suppose first that $G$ is 2-factor Hamiltonian in addition to being cubic and bipartite, which means that it is 3-connected and matching covered due to \Cref{lem:3con} and \Cref{thm:matcov}.
    Let $\Cut{X} = \{ e,f,g \}$ be the principal 3-edge cut in $G$ with $X \subseteq V(G_1)$.
    Using \Cref{lem:starproductmatching} and \Cref{lem:tightcutshape}, we note that $\Cut{X}$ is a non-trivial tight cut.
    Thus for each $x \in \Cut{X}$ there exists a perfect matching $M_x$ in $G$ with $M \cap \Cut{X} = \{ x \}$.
    Therefore $G - M_x$ is a 2-factor containing the two remaining edges of $\Cut{X}$ according to \Cref{obs:perfmat}, allowing us to apply \Cref{lem:starproductfix2} to show that $G_1$ and $G_2$ are 2-factor Hamiltonian.

    In the reverse direction we suppose that $G_1$ and $G_2$ are 2-factor Hamiltonian, but $G$ is not.
    As noted above, we know that $G_1$ and $G_2$ are both matching covered.
    Let $x,y \in \Cut{X}$ be two distinct edges in $G$.
    Since $G_1$ and $G_2$ are matching covered, \Cref{obs:perfmat} implies that for both $i \in [2]$, the graph $G_i$ contains a 2-factor $F_i$ using $x_i$ and $y_i$.
    The graph $( (F_1 - v_1) \cup (F_2 - v_2) ) \cup \{ x,y \}$ is therefore a 2-factor in $G$ that uses both $x$ and $y$.
    Thus we meet the requirement on the principal 3-edge cut posed in the statement of \Cref{lem:starproductfix2} and can apply it to show that $G$ must be 2-factor Hamiltonian as well.
\end{proof}

\end{document}